\documentclass[12pt]{article}

\usepackage{graphicx} 
\usepackage{amsthm}
\usepackage{amsmath}
\usepackage{amssymb}
\usepackage{verbatim}

\newtheorem{theorem}{Theorem}[section]
\newtheorem{lemma}[theorem]{Lemma}
\newtheorem{proposition}[theorem]{Propostition}
\newtheorem{corollary}[theorem]{Corollary}
\newtheorem{definition}[theorem]{Definition}
\newtheorem{remark}[theorem]{Remark}
\newtheorem{example}[theorem]{Example}

\newcommand{\R}{\mathbb{R}}
\newcommand{\N}{\mathbb{N}}
\renewcommand{\dh}{\dim_{\rm {H}}}
\newcommand{\db}{\dim_{\rm {B}}}
\newcommand{\dlb}{{\underline\dim}_{\rm {B}}}
\newcommand{\dub}{{\overline\dim}_{\rm {B}}}
\newcommand{\dui}{{\overline{\dim}}_{\theta}}

\newcommand{\das}{\dim_{\rm {A}}}

\newcommand{\M}{\mathcal{M}}
\newcommand{\dimup}{\overline{\dim}_{\mathrm B}}
\newcommand{\Cs}[2]{C_{#1}^{\,#2}}
\newcommand{\phirs}[2]{\phi_{#1}^{#2}}
\newcommand{\supp}{\operatorname{supp}}

\newcommand{\be}{\begin{equation}}
\newcommand{\ee}{\end{equation}}

\title{The sequence property for fractal dimensions}
\author{Kenneth Falconer and Yuyang Liu}
\date{}

\begin{document}

\maketitle

\begin{abstract}
For various definitions of fractal dimension that are finitely stable, including upper box dimension, upper intermediate dimensions and Assouad spectra, we show that, given a compact subset $E$ of a metric space, typically $\R^n$,  there is a convergent sequence of points contained in $E$ of the same dimension as $E$ itself. Moreover, under certain conditions it is possible for a sequence to witness the dimension of $E$ for many definitions of dimension simultaneously, for example in a self-affine set $E$ there is a single convergent sequence that has the same Assouad spectrum or intermediate dimension values as $E$ itself.
\end{abstract}

\section{Introduction}
\setcounter{equation}{0}
\setcounter{theorem}{0}

There are now many definitions of fractal dimension in use which may assign different values of dimension to a given set reflecting different features of the set, see for example \cite{Fal,MO,Spe}.  Given some definition `$\dim$' of a fractal dimension, a natural question to ask is, given a set $E$ in a metric space, can one always find a convergent sequence $(a_k)\subset E$ that `witnesses' the dimension of $E$, that is with the set of points of $(a_k)$ comprising a set of the same dimension as $E$ itself?

For Hausdorff dimension the answer is immediately negative since every countable set has Hausdorff dimension zero. By contrast, upper box dimension is more sensitive to point sets, and a countable set may have positive upper box dimension. 

Here we study these questions for several definitions of dimension: upper box dimension, Assouad dimension and the Assouad spectrum, upper intermediate dimensions and dimension profiles. These dimensions share two basic properties which play a central role throughout the paper: monotonicity and finite stability. These properties allow one to `localise' the dimension of a compact set.
 
In the final section we examine conditions under which a single sequence can witness several or many definitions of dimension simultaneously. For example, for  upper intermediate dimension $\overline{\dim}_{\theta}$ and $E\subset \R^n$ a self-affine set, there is a sequence $(a_k) \subset E$ such that $\overline{\dim}_{\theta}(a_k)= \overline{\dim}_{\theta}E$ for all $\theta\in (0,1]$. 

We follow a fairly consistent approach to constructing these sequences for different dimensions, using `dimension points' $x$ for which the dimension of a set is realised in all small neighbourhoods of $x$. We use the particular definition of the dimension to find point sets inside a family of balls of decreasing radii centred at $x$ which behave appropriately at decreasing scales and then take the union of these point sets to get a sequence of the required dimension. 

\section{Properties of general dimensions}
\setcounter{equation}{0}
\setcounter{theorem}{0}

In this section we note some general properties that dimensions may have and derive  a simple `localisation' lemma that exhibits a point $x \in E$ such that every ball centred at $x$ intersects $E$ in a set of full dimension. Such an $x$ is typically the limit point of the sequences which we seek.                                        

We work in a (non-empty) metric space $(X,d)$ and consider dimension functions $\dim: \mathcal{P}(X)\setminus \varnothing \to [0,\infty]$, where $\mathcal{P}(X)$ is the power set of $X$. In some settings we just need  $\dim$ to be defined on a subclass of $\mathcal{P}(X)$, for example the bounded sets or perhaps Borel or analytic sets. 

There are a number of natural properties that a dimension function $\dim$ may or may not have. In particular for subsets of a metric space,
\medskip

(1)  \textbf{Monotonicity:} $\dim$ is \textit{monotone} if $A\subset B$ implies that $\dim A \leq \dim B$,
\medskip

(2)    \textbf{Finite sets:} $\dim$ has the  \textit{finite sets} property if $\dim A =0$ whenever $A$ is finite,
\medskip

(3)    \textbf{Finite stability:} $\dim$ is  \textit{finitely stable} if $\dim(A\cup B) = \max \{\dim A,\dim B\}$,
\medskip

(4)    \textbf{Countable stability:} $\dim$ is  \textit{countably stable} if $\dim(\bigcup_i A_i) = \sup_i \dim A_i$ for all finite or countable collections of sets $\{A_i\}_i$.
\medskip

\noindent Note that if a dimension is finitely stable and has the finite sets property then the dimension of a set is unchanged on adding or removing a finite set of points. In particular:
\smallskip
\be\label{findiff}
  \qquad \text{ (i)} \dim (A\cup F) = \max\{\dim A, \dim F\} =  \dim A\quad\text{ if } F\text { is finite, }
  \ee
  $$
 \text{ (ii) For a sequence } \dim (a_k)_{k=1}^\infty = \dim (a_k)_{k=k_0}^\infty\quad\text{  for all }k_0 \in \mathbb{N}.
$$

These are the properties with which we will mainly be concerned here; there are other frequently considered attributes, such as Lipschitz or bi-Lipschitz stability and the open interior property. Most well-known dimensions, except lower dimension and its relatives, are monotonic and have the finite set property. Most are finitely stable, though lower box dimension is not. However, whilst Hausdorff and packing dimensions are countably stable, many other definitions of dimension fail to be so, including upper box dimension, intermediate dimensions, Assouad dimension and related definitions, see \cite{Fal,FZ,Fra} for discussions of these dimensions and their basic properties. 

To avoid repetition, we will assume that {\it all dimensions referred to in this paper are monotonic and have the finite sets property}.

We consider a further property that a dimension may or may not have: the sequence property holds if every set $E \subset X$ contains a convergent sequence of dimension equal to that of $E$ itself. Note that when we write $(a_k)$ for a sequence we assume that the index $k$ runs through the natural numbers $\mathbb{N}$. We also write  $(a_k) \subset E$ to mean the set $\{a_k: k\in \mathbb{N}\} \subset E$.
\medskip

\begin{definition} {\rm{\bf Sequence property}} A dimension $\dim$ on a metric space $(X,d)$ satisfies the \emph{sequence property}  if for every non-empty compact $E\subset X$ with $\dim E >0$ there is an $x\in E$ and a convergent sequence $(a_k) \subset E$ with $a_k \to x$ such that  $\dim(a_k)= \dim E$.
\end{definition} 

For subsets of $\R$ we can say a little more. 

\begin{lemma}\label{monotoneR}
Let $E\subset\mathbb{R}$ be compact with $\dim E >0$. If $\dim$  is finitely stable and satisfies the sequence property then 
 there exists a strictly increasing or a strictly decreasing convergent sequence $(b_k)\subset E$
such that
$\dim(b_k)=\dim E$.
\end{lemma}

\begin{proof}
By the sequence property, there is a convergent sequence $(a_k)\subset E$ with $a_k \to x\in E$ and
$\dim(a_k)=\dim E$. By finite stability either  $\dim((a_k)\cap (-\infty, x])=\dim E$ or $\dim((a_k)\cap [x, \infty))=\dim E$. In the former case, inductively set $b_k=\min \{a_i: a_i \neq b_j \text{ for }  1\leq j \leq k-1\}$,
so $(b_k)$ increases strictly and converges to $x$, and $\dim (b_k) =\dim(a_k)\cap (-\infty, x])=\dim E$. The latter case is similar.
\end{proof}

An obvious observation is that the sequence property fails for dimensions that are countably stable.

\begin{remark}
Let $\dim$ be a definition of dimension that is countably stable with the finite sets property. Then $\dim$ does not have the sequence property, since $\dim(a_k) = 0$ for every sequence $(a_k)$. In particular, Hausdorff and packing dimension do not have the sequence property.
\end{remark}

The sequence property also typically fails for dimensions that are not finitely stable.

\begin{remark}
Let $\dim$ be a dimension satisfying \eqref{findiff} and suppose there are disjoint compact sets $E_1,E_2\subset X$ such that 
\be\label{notfinstab}
\max \{\dim E_1, \dim E_2\}<\dim(E_1\cup E_2).
\ee
Then the sequence property fails by considering $E = E_1\cup E_2$. For since $E_1$ and $E_2$ are separated by a positive distance, any sequence $(a_k)$ in $E$ with $a_k\to x\in E$  must, for some $k_0$,  satisfy $a_k \in E_i$  for all $k \geq k_0$ either for $i=1$ or $i=2$.
Thus $\dim (a_k)_{k=1}^\infty = \dim (a_k)_{k=k_0}^\infty\leq \dim E_i<\dim E$.
\end{remark}

For example, lower box dimension does not satisfy the sequence property. It is not finitely stable, since disjoint compact sets $E_1,E_2\subset \mathbb{R}$ may be constructed satisfying  \eqref{notfinstab}, see for example, \cite[Proposition 5.5(ii)]{Ban}, but it does satisfy \eqref{findiff}.

Throughout, $B(x,r)$ will denote the closed ball of centre $x$ and radius $r$ with $B^O(x,r)$ the corresponding open ball. For $E\subset X$ we write $\# E$ for the \emph{cardinality} of $E$ and $|E| = \sup_{x,y\in E} d(x,y)$ for the \emph{diameter} of $E$.

In the forthcoming sections, we will establish that various definitions of dimensions have the sequence property. The following  simple localisation property is key to our arguments. 

\begin{lemma}\label{fullballpoint}
Let $(X,d)$ be a metric space and let $E\subset X$ be compact. Let $\dim$ be a finitely stable dimension on subsets of X. Then there exists a point $x\in E$ such that
for all $r>0$,
\be\label{dimpt}
\dim (E\cap B(x,r))= \dim E.  
\ee
\end{lemma}

\begin{proof}
Suppose that for all $x\in E$ there exists
$r_x>0$ with 
\begin{equation} \label{eq1}
\dim(E\cap B(x,r_x))<\dim E.
\end{equation}
The family $\{B^O(x,r_x):x\in E\}$ is an open cover of $E$, so by compactness there exist finitely many points
$x_1,\dots,x_m\in E$ such that
\[
E\ = \ E \cap\bigcup_{i=1}^m \bigl(B(x_i,r_{x_i})\bigr)=\bigcup_{i=1}^m \bigl(E\cap B(x_i,r_{x_i})\bigr).
\]
By finite stability and \eqref{eq1},
\[
\dim E=\max_{1\le i\le m}\dim\bigl(E\cap B(x_i,r_{x_i})\bigr)<\dim E,
\]
a contradiction that yields the result.
\end{proof}

We call a point $x\in E$ satisfying \eqref{dimpt} a {\it dimension point of }$E$. Thus for a finitely stable dimension, every compact set has at least one dimension point.
\medskip

A simple argument verifies the sequence property for upper box dimension in a metric space in Section 3 and for Assouad dimension and the Assouad spectrum in Section 4. In Section 5 we establish the sequence property for intermediate dimensions in locally compact metric spaces, via a continuity property of intermediate pre-measures. In Section 6 we use capacities to obtain results for dimension profiles in $\R^n$ which are closely related to projection properties. In Section 7 we consider when a single sequence can witness a range of dimension definitions simultaneously and show that this is so for many familiar fractals.

\section{Box dimensions}
\setcounter{equation}{0}
\setcounter{theorem}{0}

We recall the well-known definitions of the box dimensions.

\begin{definition}
Let $(X,d)$ be a metric space and let $E\subset X$ be bounded. For $r>0$ let
\be\label{box1}
N_r (E)=\{\text{\rm the least number of sets with diameters at most $r$ that can cover $E$}\}.
\ee
The \emph{lower} and \emph{upper box} or \emph{box-counting dimensions} of $E$ are given by 
 \be\label{lbox}
 \dlb E=\liminf_{r\searrow 0}\frac{\log N_r (E)}{\log(1/r)}
\ee
and
 \be\label{ubox}
 \dub E=\limsup_{r\searrow 0}\frac{\log N_r (E)}{\log(1/r)}
\ee
If $\dlb E= \dub E$ we term the common value the \emph{box dimension} $\db E$ of $E$.
\end{definition}

The values in \eqref{lbox} and \eqref{ubox} are unchanged for a number of alternative  definitions of $N_r (E)$ in \eqref{box1}, see \cite[Chapter 2]{Fal}. Particularly relevant here is the equivalent definition obtained by taking  
\begin{align}\label{box2}
N_r (E)=\{&\text{the greatest  number of points in any $r$-separated set }\nonumber \\
&\text{$S\subset E$, i.e. where $d(x,y)\geq r$ for all distinct $x,y\in S$} \}.
\end{align}

\subsection{The sequence property for upper box dimension}

We use Lemma \ref{fullballpoint} to show that upper box dimension has the sequence property. This was established in \cite{Iva} but we include a short proof here for completeness and to illustrate a more general method.

\begin{theorem}\label{seqpropbd}
Upper box dimension $\dub$ has the sequence property.
\end{theorem}

\begin{proof}
Let $E$ be a compact subset of a metric space $(X,d)$ with $\dub E=s>0$. We use the definition \eqref{ubox} of  $\dub$ taking $N_r$ as in \eqref{box2}.

We choose a sequence of scales and find large separated sets at each scale. 
Let  $x\in E$ be any dimension point as in Lemma \ref{fullballpoint}, so that
$\overline{\dim}_{\mathrm B}(E\cap B(x,\rho))=s$ for all $\rho>0$.
Set  $E_n=E\cap B(x,2^{-n})$, so $\overline{\dim}_{\mathrm B}E_n=s$ for all $n$.

First assume that $\dub E< \infty$. For each $n$ we may choose $r_n\in(0,2^{-n})$ such that
\be\label{sminus}
\frac{\log N_{r_n}(E_n)}{\log(1/r_n)}\ge s-\frac1n.
\ee
Pick an $r_n$-separated set $P_n\subset E_n$ of maximal cardinality $\# P_n=N_{r_n}(E_n)$.
Enumerate $\bigcup_{n=1}^\infty (P_n\setminus \{x\}) \subset E$ as a sequence $(a_k)$; 
since the finite set $P_n\subset B(x,2^{-n})$ it follows that $a_k\to x$. Then
$P_n\subset (a_k)\cup \{x\}$, so
\[
N_{r_n}((a_k)\cup \{x\})\ge \# P_n=N_{r_n}(E_n).
\]
Using the finite sets property of  $\dub$, 
$$
\overline{\dim}_{\mathrm B}(a_k)
=\overline{\dim}_{\mathrm B}((a_k)\cup \{x\})
=\limsup_{r\searrow 0}\frac{\log N_r ((a_k)\cup \{x\})}{\log(1/r)}
\ge \limsup_{n\to\infty}\frac{\log N_{r_n}(E_n)}{\log(1/r_n)}
\ge s,
$$The opposite inequality $\dub(a_k)\le \db E =s$ is immediate from monotonicity.

If $\dub E= \infty$ the proof is similar but with the right-hand side of \eqref{sminus} replaced by $n$.
\end{proof}

Note that in Theorem \ref{seqpropbd}, we may find a suitable sequence converging to any dimension point of $E$.

By Lemma \ref{monotoneR}, for compact subsets $E\subset \R$, we may find a monotonic convergent sequence in $E$ of upper box dimension equal to that of $E$.

\section{Assouad dimension and the Assouad spectrum}
\setcounter{equation}{0}
\setcounter{theorem}{0}

Assouad dimension might be thought of as a `local' version of box-dimension in that it depends on the  parts of a set where the box dimension is locally large. Thus the Assouad dimension of a set $E$ reflects the number of balls of radius $r$ required to cover portions of $E$ of the form $E\cap B(x,R)$ where $0<r<R$. This dimension was introduced in the late 1970s by Patrice Assouad \cite{Ass} as a tool for studying  Lipschitz-embeddability of sets in Euclidean space. From around 2010 Assouad dimension, and variants such as the Assouad spectrum, have been studied in their own right, see \cite{Fra} for a comprehensive account.

\begin{definition}\label{assdef}
Let $(X,d)$ be a metric space and let $E\subset X$ be bounded. Let $N_r(F)$ denote the minimum number of closed balls of radius $r$
needed to cover $F$.
The \emph{Assouad dimension} of $E$ is given by
\begin{align}
\das E=\inf\Big\{& s\ge 0: \text{\rm there exists } C> 0\ \text{\rm such that for all }
0<r<R \nonumber\\
& \text{\rm and } x\in X, \ 
N_{r}\bigl(E\cap B(x,R)\bigr)\le C\Bigl(\frac{R}{r}\Bigr)^s\Big\}.\label{AsIneq}
\end{align}
\end{definition}

As with box-counting dimension, the definition of Assouad dimension is unchanged for various alternative definitions of  $N_r$, including setting $N_r(F)$ to be the maximum cardinality of any $r$-separated subset of $F$, see \cite[Section 2.1]{Fra}. 

We give a short verification of the sequence property for Assouad dimension. For related results on the existence of subsets of smaller Assouad dimension see \cite{CWW}.

\begin{theorem}\label{assseq}
Assouad dimension $\das$ has the sequence property.
\end{theorem}

\begin{proof}
Let $(X,d)$ be a metric space and let $E\subset X$ be compact with $s= \das E>0$.
Since $\das  (a_k) \le \das E$ for every sequence $(a_k) $ in $E$ by monotonicity, it suffices to find a convergent sequence in $E$ such that $\das (a_k) \geq \das E$.

Assouad dimension is finitely stable so  Lemma~\ref{fullballpoint} gives an $x\in E$ such that $\das (E\cap B(x,\rho))=\das E=s$
for all $\rho>0$. Let $E_{n}=E\cap B(x,2^{-n})$, so $\das E_{n} =s$.

If $\das E <\infty$, for each $n> 1/s$ the defining inequality for Assouad dimension in \eqref{AsIneq} fails for $E_{n}$ and exponent $s-1/n$
and for every constant $C>0$. Thus, choosing $C=n$,  we can find $x_n\in E_n$ and $0<r_n<R_n$  such that
\begin{equation}\label{enbad}
N_{r_n}\bigl(E_n\cap B(x_n,R_n)\bigr)\ \ge\ n\Big(\frac{R_n}{r_n}\Big)^{s-1/n},
\end{equation}
where throughout this proof we take $N_r(F)$ to be the maximum number of points of any $r$-separated subset of $F$.
Choose an $r_n$-separated  set $P_n\subset E_n\cap B(x_n,R_n)$ with
\begin{equation}\label{defPn}
\#P_n=N_{r_n}\bigl(E_n\cap B(x_n,R_n)\bigr).
\end{equation}
 Enumerate the set $\bigcup_{n=\lceil 1/s\rceil}^\infty (P_n \setminus \{x\}) \subset\ E$ as a sequence $(a_k)$; since  $P_n\subset B(x,2^{-n})$ we conclude that $a_k\to x$.

To check that $\das ( (a_k) \cup \{x\}) =s$, let $\alpha<s$ and $C\ge 1$ be arbitrary. Choose $n$ so large that $\alpha<s-1/n$ and $n>C$.
Using \eqref{enbad}--\eqref{defPn}, we obtain
\[
N_{r_n}\bigl(( (a_k) \cup \{x\}) \cap B(x_n,R_n)\bigr)\ \ge\ N_{r_n}(P_n)= \#P_n
\ \ge\ n\Big(\frac{R_n}{r_n}\Big)^{s-1/n}
> C\Big(\frac{R_n}{r_n}\Big)^{\alpha}.
\]
Thus the Assouad covering inequality fails for $(a_k) \cup \{x\}$  taking exponent $\alpha$ and constant $C$, hence
$\das (a_k)=\das ((a_k) \cup \{x\})\ge \alpha$. Since this is true for all $\alpha<s$, we conclude that $\das (a_k) =s$.

The argument is similar when $\das E =\infty$, where in \eqref{enbad} and thereafter we replace $s-1/n$ by $n$.
\end{proof}

Closely related to Assouad dimension is the Assouad spectrum. This was introduced to provide insight into which pairs of scales $r<R$ in \eqref{AsIneq}  give rise to the extreme behaviour singled out by the Assouad dimension. Here, for  given $\vartheta\in(0,1)$, only balls of radii $R^{1/\vartheta}$ are considered in coverings of  $E\cap B(x,R)$. The Assouad spectrum is continuous for $\vartheta\in(0,1)$ and provides a `spectrum of dimensions' that interpolates between upper box dimension and Assouad dimension, see \cite{Fra,FY} .

\begin{definition}\label{assspec}
Let $(X,d)$ be a metric space and let $E\subset X$ be bounded. Let $N_r(F)$ denote the minimal number of closed balls of radius $r$
needed to cover $F$.  The \emph{Assouad spectrum} of $E$ at $\vartheta$, where $\vartheta\in(0,1)$, is defined by
\begin{align*}
\das^\vartheta E=\inf\Big\{& s\ge 0: \text{\rm there exists } C> 0\ \text{\rm such that for all }
0<R<1 \nonumber\\
& \text{\rm and } x\in E, \ 
N_{R^{1/\vartheta}}\bigl(E\cap B(x,R)\bigl) \le C\Bigl(\frac{R}{R^{1/\vartheta}}\Bigr)^s\Bigr\}.
\end{align*}
\end{definition}

\begin{theorem}\label{secpropas}
The Assouad spectrum $\das^\vartheta$ has the sequence property for each $\vartheta\in(0,1)$.
\end{theorem}

\begin{proof}
The proof proceeds in a very similar way to that of Theorem \ref{assseq}  by setting $r_n = R_n ^{1/\vartheta}$ throughout, noting that the Assouad spectrum is finitely stable and monotone for each $\vartheta\in(0,1)$.
\end{proof}

\section{Intermediate dimensions}
\setcounter{equation}{0}
\setcounter{theorem}{0}

Intermediate dimensions were introduced  \cite{FFK} for bounded subsets of $\R^n$ to interpolate between Hausdorff dimensions and box dimensions. For $0\leq \theta\leq 1$  the {\it upper $\theta$-intermediate dimension} of a non-empty bounded $E\subset \mathbb{R}^n$ is given by
\begin{align}\label{upint}
\dui E  =  \inf &\big\{ s\geq 0  :  \mbox{\rm for all $\epsilon >0$ there exists $r_0>0$ s.t. for all  $0<r<r_0$,}\nonumber\\
 & \mbox{ \rm there is a cover $ \{U_i\} $ of $E$ s.t. $r^{1/\theta} \leq  |U_i| \leq r$ and $\sum |U_i|^s \leq \epsilon$}  \big\}
\end{align}
(by convention $r^{1/0}=0$). The {\it lower  $\theta$-intermediate dimension }$\underline{\dim}_{\theta}E$ is defined in the same way except the conditions are only required to hold for arbitrarily small $r>0$, that is for a sequence of $r$ approaching $0$. However, we do not consider 
$\underline{\dim}_{\theta}E$ further; in a similar manner to lower box dimension it is not finitely stable nor does it satisfy the sequence condition \cite[Proposition 5.5]{Ban}.

The intermediate dimensions $\overline{\dim}_{\theta}E$  provide a spectrum of dimensions that range from  Hausdorff to box dimensions, that is  $\theta\mapsto \overline{\dim}_{\theta}E$ is non-decreasing for $\theta \in [0,1]$ and
$$\dh E = \overline{\dim}_0 E \leq \overline{\dim}_{\theta}E
\leq \overline{\dim}_{1}E  = \dub E.$$ 
The function $\theta\mapsto \overline{\dim}_{\theta}E$ is continuous on $(0,1]$ and may or may not be continuous at 0.
The basic properties of intermediate dimensions on $\R^n$ are described in \cite{Fal2,FFK} and for $\theta \in (0,1]$ behave more like box dimensions than Hausdorff dimension. 

The definition of intermediate dimensions can be extended to a wider class of metric spaces $(X,d)$, see \cite{Ban}. However, care is needed since, unlike box and Assouad dimensions, the definition of $\overline{\dim}_{\theta}E$ is not intrinsic to $E$ but also depends on the ambient space $(X,d)$: some spaces $(X,d)$ may not allow a sufficient range of sets of suitable diameters  to provide the required coverings. For example, the set $E= \{0, 1, \frac12,\frac13,\ldots\}$ as a subset of  $\R $ with the usual metric has $\overline{\dim}_{\theta}E= 1/(1+\theta)$. However, taking  $E$ itself as the ambient space with the usual distance, $\overline{\dim}_{\theta}E$ is not defined since, for example, there is no set $U \subset E$ with $0<|U|<\frac12$ that contains the point $\{1\}$.
This question was discussed by Banaji \cite{Ban} who noted that if  $(X,d)$ is a uniformly perfect metric space then $\overline{\dim}_{\theta}E$ is defined for $E\subset X$ and the intermediate dimensions have the properties mentioned above for $\R^n$ including continuity in $\theta \in (0,1]$. (A space is {\it uniformly perfect } if there is $c\in (0,1)$ such that $B(x,r)\setminus B(x,cr) \neq \varnothing$ for all $x\in X$ and $0<r\leq |X|$; intuitively the space does not have islands that are too separated from the rest of the space.)

\subsection{The sequence property for upper intermediate dimensions}

Showing  that the upper intermediate dimension $\overline{\dim}_{\theta}$ satisfies the sequence condition for a $\theta \in (0,1]$, is  more involved than for upper box dimension since brackets of scales need to be considered. Here we give a proof in the case where the ambient space is a locally compact metric space, this includes the Euclidean spaces $\R^n$.  

It is convenient to introduce a definition of upper intermediate dimensions equivalent to \eqref{upint} using notation reminiscent of the pre-measures in the definition of Hausdorff measures.

For a non-empty $E\subset X$ and $0<r\leq r'$ we call a finite or countable collection $\{U_i\}_i$ of subsets of $ X$ an $(r,r')$-{\it cover} of  $E$ if $r\leq |U_i|\leq r'$ for all $i$ and $E\subset \bigcup_i U_i$. We emphasise that we allow covering by subsets of the ambient space $X$ and not just by subsets of $E$.

For $\alpha>0, \ 0<r\leq r'<1$ and $E \subset X$ we define the \emph{intermediate pre-measures} by 
\be\label{intpre}
\mathcal{H}_{r, r'}^\alpha(E) = \inf \Big\{ \sum_i |U_i|^\alpha : \{U_i\} \text{ is an } (r, r')\text{-cover of } E\Big\}.
\ee
It follows from \eqref{upint} that
\be\label{equivint}
\overline{\dim}_{\theta}E  =  \inf \big\{ s\geq 0  :   \limsup_{r \to 0} \mathcal{H}_{r^{1/\theta},r}^{s}(E) =0\big\},
\ee
which is equivalent to
\be\label{equivint2}
\overline{\dim}_{\theta}E  =  \sup \big\{ s\geq 0  :   \limsup_{r \to 0} \mathcal{H}_{r^{1/\theta},r}^{s}(E) =\infty\big\}.
\ee
We write $\mathcal{C}(X)$ for the family of all non-empty compact subsets of the metric space $(X,d)$. Recall that the Hausdorff metric $d_H$ is defined on $\mathcal{C}(X)$ by 
$$d_H(A,B) = \inf\{\delta: A \subset B_\delta \text{ and }  B \subset A_\delta\},\quad A,B \in \mathcal{C}(X),$$
where $A_\delta = \big\{x: \inf\{d(x,a): a\in A\} \leq \delta\big\}$ is the $\delta$-neighbourhood of $A$ in $(X,d)$. 

We first show that in a locally compact space intermediate pre-measures are continuous from below in the Hausdorff metric, and then use this to derive the sequence property.

\begin{lemma}\label{imdimcont}
Let $E$ be a non-empty compact subset of a compact metric space $(X,d)$ and let $0<r< r'$ and $\alpha>0$ be such that $\mathcal{H}_{r, r'}^\alpha(E)$ exists with  $\mathcal{H}_{r, r'}^\alpha(E)<\infty$. Let $\{E_k\}_{k=1}^\infty $ be a sequence of non-empty compact sets such that $E_k \subset E$ for all $k$ and $E_k \to E$ in the Hausdorff metric. Then 
$\mathcal{H}_{r, r'}^\alpha(E_k )\to \mathcal{H}_{r, r'}^\alpha(E)$.
\end{lemma}

\begin{proof}

In \eqref{intpre} it is enough to take the infimum over $(r,r')$-covers $\{U_i\}$ of $E$ with at most $m:=\big\lceil \mathcal{H}_{r, r'}^\alpha(E) r^{-\alpha} \big\rceil$ sets. Also, as $\mathcal{H}_{r, r'}^\alpha(E_k) \leq \mathcal{H}_{r, r'}^\alpha(E)$, we need only consider $(r,r')$-covers of $E_k\subset E$ by at most $m$ sets for each $k$. We will repeatedly take infinite subsequences of $\{E_k\}$ and corresponding $(r,r')$-covers.

Firstly, let $\{E_{k_1}\}$ be a subsequence of $\{E_k\}$ such that 
\be\label{sublim}
\lim_{k_1 \to \infty}\mathcal{H}_{r, r'}^\alpha(E_{k_1})= \liminf_{k\to \infty}\mathcal{H}_{r, r'}^\alpha(E_k).
\ee
Let $\mathcal{U}_{k_1}= \{U_{{k_1},i}\}_{i=1}^{m_{k_1}}$ be an $(r,r')$-cover of $E_{k_1}$ by closed sets for each $k_1$, where $m_{k_1} \leq m$ and 
\be\label{sums}
0< \mathcal{H}_{r, r'}^\alpha(E_{k_1})\leq \sum_{i=1}^{m_{k_1}} |U_{{k_1},i}|^\alpha < \mathcal{H}_{r, r'}^\alpha(E_{k_1})+1/{k_1}.
\ee
We may find a subsequence $\mathcal{U}_{k_2} = \{U_{k_2,1}, U_{k_2,2} \ldots, U_{k_2,m_0} \}$ of the $m_k$-tuples of sets $\mathcal{U}_{k_1}$ for which $m_k = m_0$ is constant for all $k$. Since $(X,d)$ is compact, the space $\mathcal{C}(X)$ of non-empty compact subsets of $X$ is sequentially compact in the Hausdorff metric, see, for example,  \cite[Section 2.4]{Edg}. By repeatedly taking convergent subsequences of  
$\mathcal{U}_{k_2} = \{U_{k_2,1}, U_{k_2,2} \ldots, U_{k_2,m_0} \}$ to obtain convergence in the Hausdorff metric of each term in $\mathcal{U}_{k_2}$ in turn we may obtain an $m_0$-tuple $\mathcal{U}_{k_3} = \{U_{k_3,1}, U_{k_3,2} \ldots, U_{k_3,m_0} \}$ such that $U_{k_3,i}\to U_i$ in the Hausdorff metric $d_H$ for some $U_i \in \mathcal{C}(X)$ for each $1 \leq i \leq m_0$. 
 
 Let $x\in E$. As $E_k \to E$ in $d_H$ there is a sequence $(x_k)$ with $x_k \in E_k$ such that $x_k \to x$, and so some $1\leq i\leq m_0$ with  $x_{k_3} \in U_{k_3,i}$ for infinitely many $k_3$; call this subsequence  $(x_{k_4})\in U_{k_4,i}$,  with $x_{k_4} \to x$. As $U_{k_4,i}\to U_i$ in the Hausdorff metric, there are points $y_{k_4} \in U_{i}$ such that $d(x_{k_4},y_{k_4}) \to 0$, so $y_{k_4} \to x$, giving that $x \in U_i$ since $U_i$ is closed. It follows that $E \subset \bigcup_{i=1}^{m_0}U_i$. 
Since diameter is continuous in the Hausdorff metric, $r\leq |U_i|\leq r'$ for each $i$, so $\{U_i\}_{i=1}^{m_0}$ is an $(r,r')$-cover of $E$. 

Using \eqref{sums}, \eqref{sublim} and that $E_k \subset E$,
\begin{align*}\label{sumineq}
\mathcal{H}_{r, r'}^\alpha(E) \leq \sum_{i=1}^{m_0} |U_{i} |^\alpha = &\lim_{k_3\to \infty}\sum_{i=1}^{m_0} |U_{k_3,i} |^\alpha  =\lim_{k_3\to \infty}\mathcal{H}_{r, r'}^\alpha(E_{k_3})\\
&= \liminf_{k\to \infty}\mathcal{H}_{r, r'}^\alpha(E_k)\leq \limsup_{k\to \infty}\mathcal{H}_{r, r'}^\alpha(E_k)\leq \mathcal{H}_{r, r'}^\alpha(E),
\end{align*}
giving the conclusion.
\end{proof}

We can now obtain a sequence result for intermediate dimensions in locally compact spaces.

\begin{theorem}\label{imdimpf}
In a locally compact metric space $(X,d)$, for each $\theta\in(0,1)$ the intermediate dimension $\dui$ has the sequence property  for compact sets $E\subset X$ for which $\dui E$ exists with $\dui E<\infty$.
\end{theorem}
\begin{proof}
For a compact set $E$ in a locally compact metric space there exists $\delta>0$ such that the $\delta$-neighbourhood $E_\delta$ of $E$ is compact. For finding $\dui A$ for $A\subset E$ it is enough to consider $\mathcal{H}_{r^{1/\theta}, r}^\alpha (A)$  for $r\leq \delta$ so we may work entirely in the compact metric space $(X,d) = (E_\delta,d)$.

Let  $0<\dui E = s<\infty$. If $\alpha >s$ then  
$\limsup_{r \to 0} \mathcal{H}_{r^{1/\theta},r}^{\alpha}(E) =0$, so  $\mathcal{H}_{r^{1/\theta},r}^{\alpha}(E) \leq 1$ for all $0<r\leq r_0$ for some $0< r_0<\min\{1,\delta\}$.  From \eqref{intpre} $E$ has an $(r^{1/\theta},r)$-cover comprising at most $\big\lceil r^{-\alpha/\theta} \big\rceil$ sets. In particular,  all subsets of $E$ have a finite $(r^{1/\theta},r)$-cover for all $0<r\leq r_0$.

Let $x$  be a dimension point of $E$, as in Lemma~\ref{fullballpoint}, so $\dui(E\cap B(x,\rho))=s$ for all $\rho>0$. Let $E_n=E\cap B(x,2^{-n})$ so $\dui E_n=s$ for all $n\in \mathbb{N}$. 
For each $n> 1/s$ choose $r_n\in(0,2^{-n})$ such that $r_n < r_{n-1}$ and
\be\label{hsen}
\mathcal{H}_{r_n^{1/\theta}, r_n }^{s-1/n}(E_n) \geq 1
\ee
using \eqref{equivint2}. As $E_n$ is compact it is separable. Thus  we may find a sequence $(a_i^n)_{i=1}^\infty$ that is dense in $E_n$, so the sequence of compact sets $A_k^n = \bigcup_{i=1}^k \{a_i^n\}\subset E_n$ converges to $E_n$ in the Hausdorff metric. By Lemma \ref{imdimcont}  $\mathcal{H}_{r_n^{1/\theta},r_n}^{s-1/n}(A_k^n)\to \mathcal{H}_{r_n^{1/\theta},r_n}^{s-1/n}(E_n)\geq 1$ as $k\to\infty$, so we may choose some $k$ such that $\mathcal{H}_{r_n^{1/\theta}, r_n }^{s-1/n}(A_k^n) \geq \frac12$ to obtain a finite set $P_n := A_k^n\subset E_n$.

Let $n_0>1/s$ be such that $0<r_{n_0} \leq r_0$. Let $P = \bigcup_{n=n_0}^\infty P_n \subset E$. For $0<\alpha <s$, 
$$\mathcal{H}_{r_n^{1/\theta}, r_n }^\alpha(P)\geq\mathcal{H}_{r_n^{1/\theta}, r_n }^\alpha(P_n)\geq \mathcal{H}_{r_n^{1/\theta}, r_n }^{s-1/n}(P_n)\geq \textstyle{\frac12}$$
if $s-1/n > \alpha$.
This holds for a sequence $r_n\to 0$ so $\alpha \leq \dui P\leq \dui E =s$  using \eqref{equivint} and monotonicity. Taking $\alpha$ arbitrarily close to $s$ gives $\dui P=\dui E $.
Enumerating  $P\setminus \{x\} \subset E$ gives a sequence $(a_k)$ converging to the only accumulation point $x$ of $P$ since $P_n\subset B(x,2^{-n})$ for each $n$, so   $\overline{\dim}_{\theta}(a_k)=\dui P =\dui E$.
\end{proof}

Again we remark that in Theorem \ref{imdimpf} we may find a suitable sequence converging to any dimension point of $E$.

\section{Box dimension profiles -- Projection theorems}\label{dimprof}
\setcounter{equation}{0}
\setcounter{theorem}{0}
For $1\leq m< n$ let $G(n,m)$ be the Grassmanian of $m$-dimensional subspaces of $\mathbb{R}^n$ endowed with the natural invariant measure. Let $\pi_V:\mathbb{R}^n\to V$ denote orthogonal projection onto $V\in G(n,m)$. The well-known projection theorem of Marstrand for sets in $\R^2$  \cite{Mar}, extended to higher dimensions by Mattila \cite{Mat,Mat2}, states that for $E$ a Borel set in $\R^n$, $\dh \pi_V E= \min\{m, \dh E\}$ for almost all $V\in G(n,m)$. The analogous projection results for box and packing dimensions are more subtle. For a given bounded $E\subset \R^n$, 
$\dub \pi_V E$ is constant for almost all $V\in G(n,m)$; however this constant may take any value in the range  
\begin{equation}\label{boxbounds} 
\frac{\dub E} {1+ (1/m - 1/n)  \dub E} \leq  \dub \pi_V E\leq \min\{\dub E,m\},
\end{equation}
and examples show that these bounds are sharp \cite{FH,Jar}. The almost sure values of $\dub \pi_V E$ form the  upper box {\it dimension profile} $\dub^{\,m}E$ of $E$, thought of as the upper box dimension of $E$ when viewed from an $m$-dimensional subspace. These profiles extend to when $m$ is not an integer in a natural way.
The original definition of dimension profiles \cite{FH} was cumbersome and  the definition and properties were reworked in \cite{FalCap,Fal4} in terms of capacities with respect to certain kernels, the approach we consider here.

\subsection{Capacities}
We summarise the properties of capacities that we need.

\begin{definition}\label{defintdim}
For $s>0$ and $r>0$   let $\phirs{r}{s}$ be the continuous and bounded kernel given by
\be\label{ker}
\phirs{r}{s}(x)=\min\!\Big\{1,\Big(\frac{r}{|x|}\Big)^s\Big\}
\qquad (x\in\R^n).
\ee
The \emph{energy} of a Borel probability measure $\mu$ with respect to $\phirs{r}{s}$ is 
\[
I_r^s(\mu)=\iint \phirs{r}{s}(x-y)\,d\mu(x)\,d\mu(y).
\]
Let $E\subset \R^n$ be non-empty and compact. Writing $\M(E)$ for the Borel probability measures supported by $E$, the \emph{capacity of} $E$ with respect to $\phirs{r}{s}$ is given by
\[
\frac{1}{\Cs{r}{s}(E)}=\inf_{\mu\in \M(E)} I_r^s(\mu).
\]
Note that the infimum energy will be finite for a bounded kernel on $\R^n$ and is attained by an equilibrium measure. The \emph{$s$-dimensional upper box dimension profile} of $E\subset \R^n$ is then defined by 
\be\label{udpdef}
\dub^{\,s}E =\limsup_{r\to 0}\frac{\log \Cs{r}{s}(E)}{-\log r} \qquad (s>0).
\ee
\end{definition}

The following proposition, see \cite{FalCap, Fal4}, motivates the introduction of box dimension profiles.

\begin{proposition}
Let $E\subset \R^n$ be compact. For integers $1\leq m< n$, $\dub \pi_V E= \dub^m E$ for almost all $V\in G(n,m)$. If $s\geq n$ then  $\dub^s E = \dub E$. 
\end{proposition}

By considering energies $I_r^s(E)$ it is easily seen that $\dub^s E$ is monotonic for each $S$. Moreover, if $E$ is finite consisting of $k$ distinct points $\{x_i\}_{i=1}^k$ and $\mu\in\M(E)$, then by Cauchy's inequality $I_r^s(\mu)\geq \sum_{i=1}^k \mu(x_i)^2 \geq 1/k$ for all $r>0$, so 
$\Cs{r}{s}(E)$ is bounded and $\dub^{\,s}E =0$.

For finite stability, let $A,B\subset \R^n$ be non-empty compact sets. Capacities with respect to many kernels, including \eqref{ker}, are subadditive, 
i.e. 
$$\Cs{r}{s}(A\cup B)\leq \Cs{r}{s}(A)+\Cs{r}{s}(B),$$
see \cite[$\S$3.6]{DP}, so
$$\log \Cs{r}{s}(A\cup B)\leq \max\{\log \Cs{r}{s}(A),\log \Cs{r}{s}(B)\}+\log 2 $$
and dividing by $-\log r$ and taking the upper limits establishes finite stability of  $\dub^{\,s}$.

\subsection{The sequence property for dimension profiles}

We use capacities to obtain the sequence property for dimension profiles $\dub^s$ for $s>0$.

\begin{lemma}\label{finapprox}
Let $s>0$ and $r>0$.   Let $F\subset \R^n$ be non-empty and compact and $0<\lambda<1$. Then there exists a finite set $P\subset F$ such that
\[
\Cs{r}{s}(P)\geq \lambda \,\Cs{r}{s}(F).
\]
\end{lemma}

\begin{proof}
Let $\mu_0\in \M(F)$ be an equilibrium measure on $F$, so that
\[
I_r^s(\mu_0)=\frac{1}{\Cs{r}{s}(F)}.
\]
Because $F$ is compact and the kernel $\phirs{r}{s}(x-y)$ is continuous and bounded on $F\times F$, the map $\M(F)\to \mathbb{R}^+$ given by 
\[
\mu\mapsto I_r^s(\mu)
\]
is continuous with respect to weak convergence of probability measures. Finitely supported probability measures are weakly dense in $\M(F)$, see for example \cite[Theorem 6.3]{Pat}, so there exists a finitely supported probability measure $\nu\in\M(F)$ such that
\[
I_r^s(\nu)\leq \frac{1}{\lambda}I_r^s(\mu_0)
\]
giving  the conclusion with $P=\supp \nu$.
\end{proof}

\begin{theorem}
For each $s>0$, the dimension profile $\dimup^{\,s}$ satisfies the sequence property on subsets of $\R^n$.
\end{theorem}

\begin{proof}
Let $E\subset \R^n$ be compact with $\dimup^{\,s}=t>0$.
As  $\dimup^{\,s}$ is finitely stable, by Lemma \ref{fullballpoint} there is a dimension point $x\in E$ such that for all $\rho>0$,
\[
\dimup^{\,s}E= \dimup^{\,s}(E\cap B(x,\rho)).
\]
Let $E_n=E\cap B(x,2^{-n})$, then for each $n>1/t$  we may choose $r_n\in(0,2^{-n})$ such that
\be\label{tminus}
\frac{\log \Cs{r_n}{s}(E_n)}{-\log r_n}> t-\frac1n.
\ee
By  Lemma \ref{finapprox} there is a finite set $P_n\subset E_n$ such that
\[
\Cs{r_n}{s}(P_n)\geq \textstyle{\frac12}\,\Cs{r_n}{s}(E_n).
\]
Since each $P_n\subset B(x,2^{-n})$ we may enumerate $\big(\bigcup_{n=\lceil 1/t\rceil}^\infty P_n)\setminus \{x\} \subset E$ as a sequence $(a_k)$ with $a_k\to x$. 
Moreover, for each $n\in\mathbb N$,
\[
\Cs{r_n}{s}\big((a_k)\cup\{x\}\big)\geq \Cs{r_n}{s}(P_n)\geq \textstyle{\frac12}\,\Cs{r_n}{s}(E_n),
\]
so taking logs, dividing by $-\log r_n$ and using \eqref{tminus}, letting $n\to\infty$ gives that 
$\dimup^{\,s}(a_k)=\dimup^{\,s}((a_k)\cup\{x\})\geq \dimup^{\,s}E$, with the opposite inequality following from monotonicity of $\dimup^{\,s}$.
\end{proof}

\section{Simultaneous representation by sequences}
\setcounter{equation}{0}
\setcounter{theorem}{0}

It is natural to ask is whether, given a compact set $E$, one can find a single convergent sequence $(a_k) \subset E$ with $\dim (a_k) =\dim E$ simultaneously for two or more definitions of dimension. In general this is not possible unless the definitions  have a common dimension point. We first give some examples that illustrate this and then we exhibit some settings where simultaneous representation is possible in a strong sense.

First, an example of a set $E$ for which, in particular, no sequence can witness Assouad spectrum values $\dim_A^{1/4}E$  and $\dim_A^{3/4}E$ simultaneously.

\begin{example}{\rm (Assouad spectrum)}\label{eg3}
There exists a compact set $E \subset [0,3]$ such that for every convergent sequence $(a_k)_k \subset E$,
\be\label{ineqass}
\dim_A^{\vartheta}(a_k)_k < \dim_A^{\vartheta}E
\ee
either for all $0<\vartheta<\frac12$ or for all $\frac12<\vartheta<1$.
\end{example}

\begin{proof}
Let $0<s<1$ and let $E_1\subset [0,1]$ be an $s$-Ahlfors regular compact set, for example a self-similar set satisfying the strong separation condition. By \cite[Section 6.4]{Fra}
$$  \dim_A^\vartheta E_1=s
    \qquad\text{for all } \vartheta\in(0,1).$$
By Rutar's characterisation of admissible Assouad spectra \cite{Rut}, there exists a compact set $E_2\subset [2,3]$ such that
$$ \dim_A^\vartheta E_2= \varphi(\vartheta) 
     \qquad\text{for all } \vartheta\in(0,1),
$$
where $ \varphi(\vartheta)=\min\{1,s/2(1-\vartheta)\}$. Setting   $E=E_1\cup E_2$, $$
    \dim_A^\vartheta E  = \max\{\dim_A^\vartheta E_1,\dim_A^\vartheta E_2\}
    =
    \max\{s,\varphi(\vartheta)\}
$$
by finite stability of the Assouad spectrum. In particular,
$$
    \dim_A^\vartheta E=s
    \ \text{ for } 0<\vartheta<\textstyle{\frac12} \quad \text{ and }\quad 
 \dim_A^\vartheta E=\varphi(\vartheta)>s
    \  \text{ for }\textstyle{\frac12}<\vartheta<1.
$$
Let $(a_k)_{k}\subset E$ with $a_k \to x \in E$. If $x \in E_1$ then $a_k \in E_1$ for $k\geq k_0$ for some $k_0$, so 
$$\dim_A^{\vartheta}(a_k)_{k\geq 1}=\dim_A^{\vartheta}(a_k)_{k\geq k_0}\leq\dim_A^{\vartheta}E_1
    =s<\varphi(\vartheta)=\dim_A^{\vartheta}E\quad \text{ if } \textstyle{\frac12}<\vartheta<1.$$
Similarly if $x\in E_2$ then
    $$\dim_A^{\vartheta}(a_k)_{k\geq 1}=\dim_A^{\vartheta}(a_k)_{k\geq k_0}\leq\dim_A^{\vartheta}E_2 =\varphi(\vartheta) <s =\dim_A^{\vartheta}E \quad \text{ if } 0<\vartheta <\textstyle{\frac12}.
$$
Thus in both cases there is a range of $\theta$ with strict inequality in \eqref{ineqass}.
\end{proof}

A similar approach provides an example for intermediate dimensions $ \overline{\dim}_\theta$.

\begin{example}{\rm (Intermediate dimensions)}\label{eg1}
There exists a compact set  $E \subset [0,3]$ such that for every convergent sequence
$(a_k)_{k}\subset E$, 
\be
\label{ineqtheta}
\overline{\dim}_\theta (a_k)_{k} < \overline{\dim}_\theta E
\ee
either for all $0<\theta<\frac13$ or for all $\frac13<\theta\leq 1$. 
\end{example}

\begin{proof}
Let $E_1\subset [0,1]$ be a compact set such that
$$ \dh E_1=\dub E_1 =\textstyle{\frac12}, $$
so $\overline{\dim}_\theta E_1 =\textstyle{\frac12}$ for all $\theta\in [0,1]$;  the middle-half Cantor set is such a set.
Let
$$ E_2:=\{2+n^{-1/3}:n\in\mathbb N\}\cup\{2\}\subset [2,3].$$
Then
$$
\overline{\dim}_\theta E_2 =\frac{3\theta}{3\theta+1}
\qquad\text{for all }\theta\in(0,1],
$$
see \cite{Fal2,FFK}.
Let
$$
E=E_1\cup E_2.
$$
By finite stability,
$$
\overline{\dim}_\theta E
=
\max\Big\{\frac12,\frac{3\theta}{3\theta+1}\Big\}.
$$
As in Example \ref{eg3}, if $(a_k)_{k}\subset E$ is a convergent sequence with $a_k \to x \in E$, then either $x \in E_1$ for all sufficiently large $k$ so
$$\textstyle{\overline{\dim}_\theta (a_k)\leq \overline{\dim}_\theta E_1 = \frac12<   \frac{3\theta}{3\theta+1} =\overline{\dim}_\theta E}\quad \text{ if } \frac13<\theta\leq 1,$$
or $x \in E_2$ so
$$\textstyle{\overline{\dim}_\theta (a_k)\leq \overline{\dim}_\theta E_2 \le \frac{3\theta}{3\theta+1}< {\textstyle \frac12}=\overline{\dim}_\theta E}\quad \text{ if } 0<\theta <\frac13.$$
\end{proof}
The following example exhibits a subset of $\R^2$ for which no convergent sequence can witness both the box dimension of the set and that of its projections onto lines. The example could equally well be expressed in terms of the dimension profiles discussed in Section \ref{dimprof}, recalling that,  for $F \subset \R^2$,  $\dub^2 F = \dub F$ and $ \dub^1 F = \dub\,{\pi_\theta}F$ for almost all $\theta \in (-\pi, \pi],$ where $\pi_\theta$ denotes orthogonal projection from $\R^2$ onto the line in direction $\theta$.

\begin{example}{\rm (Projections/dimension profiles)}\label{eg2}
There exists a compact set  $E \subset \R^2$ such that for every convergent sequence
$(a_k)_{k}\subset E$ either 
$$\dub (a_k) < \dub E $$
or
$$ \dub\, {\pi_\theta}(a_k) <  \dub\,{\pi_\theta}E \quad \text{ for almost all }\theta \in (-\pi, \pi].$$
\end{example}

\begin{proof}
Let $E_1 \subset \R^2$ be a compact set such that 
$$\dub\, E_1 =  \dub\, {\pi_\theta}E_1 =\textstyle{\frac34}$$ 
for all $\theta$, for example $E$ could be a strongly separated self-similar set with dense rotations. Let $E_2 \subset \R^2$ be  compact and disjoint from $E_1$ such that 
$$\dub\, E_2 =  1 \text{ and } \dub\, {\pi_\theta}E_2 =\textstyle{\frac23}$$ 
for almost all $\theta$, such a set exists as the extreme case of \eqref{boxbounds} taking $n=2$ and $m=1$, see \cite{FalCap,FH,Jar}.
Then setting $E=E_1\cup E_2$, 
$$ \dub\, E = \max\{\textstyle{\frac34},1\} =1 \text{ and  }  \dub\, {\pi_\theta}E= \max\{\textstyle{\frac34},\textstyle{\frac23}\}= \textstyle{\frac34}.$$

As in Example, \ref{eg1}, if $(a_k)\subset E$ is a convergent sequence with $a_k \to x\in E$ then either $a_k\in E_1$ for all sufficiently large $k$ so
Let $(a_k)_{k}\subset E$ with $a_k \to x \in E$. If $x \in E_1$ then $a_k \in E_1$ for sufficiently large $k$, so $\dub (a_k) \leq \frac34 < 1= \dub E$. Similarly, if $x \in E_2$ then 
$\dub\, {\pi_\theta}(a_k) \leq \frac23 < \frac34=  \dub\, {\pi_\theta}E$ for almost all $\theta$.
\end{proof}

Examples \ref{eg3}-- \ref{eg2} showed that if a set has no common dimension point for two definitions of dimension, there need not be a sequence that witnesses both definitions simultaneously. 
However, it may be possible to find a sequence that witnesses several different definitions of dimension if a set has  common dimension points for several dimensions such as those considered in Sections 3-6. The following lemma notes that given countably many sequences with a common limit, there exists a sequence that contains the `tail' of each of these sequences.

\begin{lemma}\label{ctble}
Let  $(X,d)$ be a metric space and let $x\in X$. For each 
$j\in\mathbb{N}$ (or for $j \in \{1,2,\ldots,n\}$), let $(a^{(j)}_k)_{k}$ be a sequence with $a^{(j)}_k \to x$.
Then there exists a sequence $(a_i)_{i}$ with $a_i \to x$ such that 
\medskip

$(i)$ $(a_i)_{i} \subset \bigcup_{j\in \N}\bigcup_{k\in \N} \{a^{(j)}_k\}$
\medskip

\noindent and
\medskip

$(ii)$ for all $j\in \N$ there is a number $k(j) \in \N$ such that $(a^{(j)}_k)_{k\geq k(j)} \subset \{a_i\}_{i\geq 1}$.

\noindent Moreover, if $(a^{(j)}_k)\neq x$ for all $j\in \N$ and $k\in \N$, then we can take $(a_i)_{i}\neq x$ for all $i\in \N$.
 \end{lemma}

\begin{proof}
For each $j\in \N$ let 
$$k(j) = \min\{i: a_k^{(j)} \in B(x,2^{-j}) \text{ for all } k\geq i\}.$$
Let $A = \bigcup_{j\in \N} \{ a_k^{(j)}: k\geq k(j)\}$. Then, noting that only finitely many points of $A$ lie outside $B(x,2^{-j}) $ for each $j$, the set $A$ is a countable collection of points with  a single accumulation point $x$, so we may order the points of $A$ as a sequence $(a_i )_i$ convergent to $x$. The final part of the conclusion corresponds to the case $x\notin A$. 
\end{proof}

We may  apply Lemma \ref{ctble} to dimensions of sequences.

\begin{corollary}\label{cor2}
Let  $E \subset X$. Let $\dim^{(j)}$ be a finite or countable collection of dimension definitions, that are finitely stable and have the finite set property, and have the sequence property with a common dimension point $x\in E$. Then there is a sequence $(a_i)_i$ with $a_i \to x$ such that $\dim^{(j)}(a_i)_i = \dim^{(j)}E$ for all $j$.
\end{corollary}

\begin{proof}
For each $j$ there is a sequence $(a^{(j)}_k)_{k}\subset E$ with $a^{(j)}_k \to x$ and $\dim^{(j)}(a_k^{(j)})_k = \dim^{(j)}E $. By Lemma \ref{ctble} there is a sequence $(a_i)_{i} \subset E$ with $a_i\to x$ and numbers $k(j) \in \N$ such that $(a^{(j)}_k)_{k\geq k(j)}\subset \{a_i\}_{i}$. Then for each $j$
$$\dim^{(j)} E\geq \dim^{(j)} (a_i)_{i} \geq  \dim^{(j)} (a^{(j)}_k)_{k\geq k(j)} = \dim^{(j)} (a^{(j)}_k)_{k\geq 1} = \dim^{(j)} E,$$
since, by finite stability and the finite set property, removing a finite number of points from a sequence does not change its dimension.
\end{proof}

For many well-known fractals, every point of $E \subset X$  is a dimension point for some or all of the dimensions considered in Sections 3-6, allowing Corollary \ref{cor2} to be applied. In particular, this will be the case if $E$ is {\it dimension homogeneous} in the sense that every ball $B(x,r)$ with $x\in E$ and $r>0$ contains a bi-Lipschitz image of $E$, so that $\dim (E\cap B(x,r)) = \dim E$ making every point of $E$  a dimension point for a wide range of dimension definitions. In particular, this is the case where $E$ is the attractor of an iterated function system (IFS) of bi-Lipschitz contractions on $\R^n$, including  self-similar, self-affine and many other fractal attractors of IFSs or graph-directed IFSs. We illustrate this with two examples relating to the intermediate dimension and dimension profile dimension homogeneous sets.

\begin{corollary}\label{intsq}
Let $E \subset \R^n$  be a dimension homogeneous compact set and let $x\in E$. 

(i) (Intermediate dimensions) There is a sequence $(a_i)_{i} \subset E$ such that $a_i \to x$ and $\overline{\dim}_\theta (a_i)_{i} = \overline{\dim}_\theta E$ for all $\theta \in (0,1]$.

(ii) (Assouad spectrum) There is a sequence $(a_i)_{i} \subset E$ such that $a_i \to x$ and $\dim_A^{\vartheta} (a_i)_{i} = \dim_A^{\vartheta} E$ for all $\vartheta \in (0,1)$.
\end{corollary}

\begin{proof}
Let $\{\theta_j\}_j$ be a countable dense subset of $(0,1]$. By Theorem \ref{imdimpf}, for each $j$ there exists a sequence $(a^{(j)}_k)_{k}\subset E$ with $a^{(j)}_k \to x$ and $\overline{\dim}_{\theta_j }(a^{(j)}_k) = \overline{\dim}_{\theta_j } E$. By Corollary \ref{cor2} there is a sequence $(a_i)_{i}$ with $a_i \to x$ such that 
\be\label{thetaj}
\overline{\dim}_{\theta_j} (a_i) = \overline{\dim}_{\theta_j }E
\ee 
for all $j$. But $\overline{\dim}_{\theta}F$ is continuous in $\theta$ for bounded $F\subset \R^n$, see \cite{FFK}, so \eqref{thetaj} extends by continuity to all $\theta\in (0,1]$.

The argument for the Assouad spectrum is similar.
\end{proof}

In Corollary \ref{intsq} we may take  $E$ to be a Bedford-McMullen self-affine carpet, see \cite{Fra4}. The intermediate dimension function $\theta\mapsto  \overline{\dim}_\theta E$ of such carpets is extraordinarily complicated \cite{BK}. This corollary demonstrates that such complexity of intermediate dimensions can be realised by a single convergent sequence.

We can obtain a similar conclusion for dimension profiles and the dimensions of projections of sets.

\begin{corollary}{\rm (Dimension profiles)}\label{eg4}
Let $E \subset \R^n$  be a dimension homogeneous compact set and let $x\in E$. Then there is a sequence $(a_i)_{i} \subset E$ such that $a_i \to x$ and $\overline{\dim}^s(a_i)_{i} = \overline{\dim}^s E$ for all $ 0<s\leq n$. In particular, for all integers $1\leq m \leq n$,  $\dub\pi_V (a_i) = \dub \pi_V E =  \dub^m E$ for almost all $V\in G(n,m)$, with $\dub^n (a_i) = \dub E$.
\end{corollary}
\begin{proof}
The proof is similar to that of  Corollary \ref{intsq}
\end{proof}

Note that for sets and dimensions without a common dimension point it may still be possible to find a countable subset having the same dimensions as the set. For example in the case of intermediate dimensions, given a compact $E \subset \R^d$, for a countable dense set $\theta_j \in (0,1]$, there exists some convergent sequence $(a^{(j)}_k)\subset E$ with 
\be\label{denseeq}
\overline{\dim}_{\theta_j }(a^{(j)}_k) = \overline{\dim}_{\theta_j } E
\ee
 where $\overline{\dim}_{\theta}$ are the intermediate dimensions. Then $P=\bigcup_{j\in \N}\bigcup_{k\in \N} \{a^{(j)}_k\}$ is a countable set containing every sequence $(a^{(j)}_k)$. Moreover, for all $\theta \in (0,1]$ there are $\theta_j \to \theta$, so by continuity of  $\overline{\dim}_{\theta} E$ and $\overline{\dim}_{\theta} P $ \eqref{denseeq} extends to all $\theta\in (0,1].$
 \bigskip
 
\begin{center}{\small ACKNOWLEDGEMENT}
\end{center}
The authors thank Amlan Banaji for helpful comments relating to this work.

\begin{flushleft}
\begin{small}
\medskip

Kenneth Falconer

Mathematical Institute, University of St~Andrews, St~Andrews, Fife, KY16~9SS, UK

{\texttt kjf@st-andrews.ac.uk}
\medskip

Yuyang Liu

Mathematical Institute, University of St~Andrews, St~Andrews, Fife, KY16~9SS, UK

{\texttt cw269@st-andrews.ac.uk}
\end{small}
\end{flushleft}

\end{document}